\documentclass[12pt]{amsart}
\usepackage{amsmath,amssymb,amscd,amsthm, enumerate,graphicx}

\def\F{\mathbb F}

\def\P{\mathbb{P}}

\newcommand{\fq}{\mathbb{F}_q}

\theoremstyle{plain}

\newtheorem{theorem}{Theorem}[section]
\newtheorem{lemma}[theorem]{Lemma}

\newtheorem{corollary}[theorem]{Corollary}
\newtheorem{question}[theorem]{Question}

\author{Amanda Knecht}
\address{Department of Mathematics and Statistics\\
Villanova University \\
Villanova, PA}
\email{Amanda.Knecht@villanova.edu}

\author{K. Reyes}
\address{Department of Mathematic\\ 
University of Michigan\\
 Ann Arbor, MI}
\email{kgre@umich.edu}

\title[Full Degree Two del Pezzo Surfaces over  $F_q$]{Full Degree Two del Pezzo Surfaces  over Finite Fields}
\begin{document}
  \maketitle
  
  
  \begin{abstract} Hirschfeld classified split del Pezzo surfaces of degree at least three whose points are all contained on the lines in the surface.  We continue his work and begin the classification of  split degree two del Pezzo surfaces over finite fields whose points are all  on the fifty-six exceptional curves of the surfaces. \end{abstract}

\section{Introduction} \label{knecht:seclabel1}

A smooth two dimensional projective variety $X$ defined over a field $k$ is called a \textit{del Pezzo} surface if  its  anticanonical divisor $-\omega_X$ is ample.  The \textit{degree} $d$ of a del Pezzo surface  is  the self-intersection number  of its canonical class and is always between 1 and 9.  Over algebraically
closed fields, del Pezzo surfaces admit a description that is simple to understand and
is adequate for the purposes of this paper. \cite[IV.24.4]{MR833513} \cite[Thm 1]{MR579026}.
\begin{theorem}[Manin, Demazure]  Let $X_d$ be a del Pezzo surface of degree $d$ defined over an algebraically closed field.
\begin{itemize}
\item[(i)] $X_9$ is isomorphic to $\P^2$.
\item[(ii)] $X_8$ is isomorphic to $\P^1 \times \P^1$ or to the blow-up of $\P^2$ at one point. 
\item[(iii)] If $1\leq d \leq 7$, then $X_d$ is the blow-up of $\P^2$ at $9-d$ points, no three of which lie on a line, no six of which lie on a conic, and no eight lie on a singular cubic having one of the points as the singularity.\end{itemize}
\end{theorem}
Because any four points in $\P^2$  can be moved to any four other points via an automorphism of $\P^2$, for $5\leq d \leq 7$, there is only one del Pezzo surface of that degree up to isomorphism \cite[IV.24.4.1]{MR833513}.  Thus, from a moduli theoretic
point of view, the most interesting del Pezzo surfaces are those of degree less than five.  Over arbitrary fields these del Pezzo surfaces may not be the blow-up of points in $\P^2$, so we need a simple description that does not require the field $k$ to be algebraically closed.  Such a description is provided by Koll{\'a}r \cite[III.3.5]{MR1440180}.
\begin{theorem}[Koll{\'a}r]  Let $X_d$ be a degree $d$ del Pezzo surface defined over an arbitrary field $k$.
\begin{itemize}
\item[(i)] $X_4$ is isomorphic to the intersection of two quadrics in $\P^4$.
\item[(ii)] $X_3$ is isomorphic to a cubic surface in $\P^3$.
\item[(iii)] $X_2$ is isomorphic to a hypersurface of degree four in weighted projective space $\P(2,1,1,1)$.
\item[(iv)] $X_1$ is isomorphic to a hypersurface of degree six in weighted projective space $\P(3,2,1,1)$.
\end{itemize}
\end{theorem}

Del Pezzo surfaces contain a finite number of surprisingly rigid curves that encode much of their geometry. They are called  \textit{exceptional curves} and are defined as  irreducible genus zero curves with self-intersection $-1$, and they are the only irreducible curves on $X_d$ with negative self-intersection.
  A surface $X_d$ is called \textit{split} over the field $k$ if it can be realized as the blow-up of $\P^2$ at $9-d$ points.   Split surfaces are of interest to us because when a surface $X_d$ is split, all of its  exceptional curves are defined over $k$.     The exceptional curves on $X_d$  can easily be described when $2\leq d\leq 7$.  
\begin{theorem}[Manin]
Let $f: X_d \rightarrow \P^2$ be the blow-up of the plane at the points $x_1, \ldots, x_{9-d}$ with $2\leq d \leq 7$ and let $L \subset X_d$ be an exceptional curve.   The image $f(L)$ in $\P^2$ is one of the following:\begin{itemize}
\item[(a)]  one of the points $x_i$;
\item[(b)] a line passing through two of the points $x_i$;
\item[(c)] a conic passing through five of the points $x_i$;
\item[(d)] a cubic passing through seven of the points $x_i$ such that one of them is a double point.
\end{itemize} 
The number of exceptional curves on $X_d$ is given in the following table:
\begin{center}
  \begin{tabular}{l| l l l l l l}
     \hline
    degree $d$& 7 & 6& 5& 4&3&2\\ \hline
    number of exceptional curves & 3 & 6 & 10& 16& 27&56 \\
    \hline
  \end{tabular}
\end{center}
\end{theorem}
For example, the 56 exceptional curves on degree 2 del Pezzo surfaces are mapped by $f: X_2 \rightarrow \P^2$ to:
\begin{enumerate}
\item[(i)] 7 points $ x_1, x_2, \ldots, x_7$ that are blown up
\item[(ii)] $21$ lines  passing through 2 of the points $x_i$ and $x_j$
\item[(iii)]  $21$ conics  passing through 5 of the points 
\item[(iv)] 7 singular cubics passing through all 7 points and having a double point at one of them.
\end{enumerate}

For the rest of the paper,   $X_d$ denotes a split degree $d$ del Pezzo surface defined over a finite field $\fq$, and  $L_{d,q}$ will denote the elements of $X_d(\fq)$ contained on  the exceptional curves of $X_d$.   Because the intersections of the exceptional curves on high degree del Pezzo surfaces are simple, an easy counting argument lets us fill in the following table when $4\leq d \leq 7$:
\begin{center}
  \begin{tabular}{| l | l| l| l| l | l| l| l| }
     \hline
    degree $d$& 7 & 6& 5& 4&3&2\\ \hline
    number of exceptional curves & 3 & 6 & 10& 16& 27&56 \\ \hline
    $L_{d,q}$ & $3q+1$ & $6q$ & $10q-5$ & 1$6q-24$  & * & * \\
    \hline
  \end{tabular}
\end{center}
There do not exist formulas depending only on $d$ and $q$  for degree 3 and 2 del Pezzo surfaces, because the intersections of the exceptional curves are  more complicated.  For cubic surfaces the exceptional curves can intersect three at a time in places called \emph{Eckardt points}.  If we let $e$ denote the number of Eckardt points on a cubic surface $X_3$ then Hirschfeld found that $L_{3,q}= 27(q-4)+e$ \cite{MR627498}.
A classical result by Andr\'e Weil gives us the number of $\fq$-rational points on split del Pezzo surfaces \cite{MR0092196}:
$$|X_d(\F_q)| = q^2 +(10-d)q+ 1.$$
Hirschfeld defined \emph{full} del Pezzo surfaces as split surfaces whose $\fq$-rational points are all contained on the exceptional curves, i.e. $|X_d(\F_q)| =  L_{d,q}.$  Combining the table above with Weil's formula leads to the following table:

\begin{center}
  \begin{tabular}{|c| c | c  | c  | }
  \hline
  degree $ d$ & $|X_d(\F_q)|$ &   $L_{d,q}$  & $|X_d(\F_q)| -  L_{d,q}.$ \\
  \hline
  7 & $q^2+3q+1$ & $3q+1$ &$ q^2$\\ \hline
  6&   $q^2+4q+1$ & $6q$ &$ (q-1)^2$ \\ \hline
  5&   $q^2+5q+1$ & $10q-5$ &$ (q-2)(q-3)$ \\ \hline
   4&   $q^2+6q+1$ & $16q-24$ &$ (q-5)^2$ \\ \hline
   3& $q^2+7q+1$ & $27(q-4)+e$ &$(q-10)^2 +9-e$\\ \hline
  \end{tabular}
\end{center}
Using the table above and an in depth analysis of  cubic surfaces, Hirschfeld classified the full del Pezzo surfaces of degree nine through three \cite{MR831145}.  
\begin{theorem}[Hirschfeld]  Let $X_d$ be a split del Pezzo surface of degree $d$ defined over a finite field $\fq$.
\begin{itemize}
\item[(i)] $X_d$ is never full when $d\in \{ 6, 7, 8, 9\}$.
\item[(ii)] $X_5$ is full if and only if $q= 2$ or $3$.
\item[(iii)] $X_4$ is full if and only if $q=5$.
\item[(iv)] Full cubic surfaces exist only over the fields $\F_4, \F_7, \F_8, \F_9, \F_{11}, 
\F_{13}, $ and $\F_{16}$.
\end{itemize}
\end{theorem}
This paper gives a partial classification of full degree two del Pezzo surfaces over fields of odd characteristic.  
\begin{theorem}
 Let $X_2$ be a full degree two del Pezzo surface over $\fq$ where $q$ is odd.   Then $9 \leq q \leq 37$ and up to isomorphism:
\begin{itemize}
\item[(i)] there is a unique $X_2$ defined over $\F_9$ and $\F_{11}$;
\item[(ii)] there are two $X_2$'s defined over $\F_{13}$;
\item[(iii)] there are six $X_2$'s defined over $\F_{17}$;
\item[(iv)] there are five $X_2$'s defined over $\F_{19}$;
\item[(v)] there are at least two  $X_2$'s defined over $\F_{23}$.
\end{itemize}
\end{theorem}
Theoretically there can exist full degree two del Pezzo surfaces over $\F_{25}, \F_{27}, \F_{29},  \F_{31},$ and $\F_{37}$, but we have not found them yet.\\\\

\noindent \textbf{Acknowledgments:}  The authors are greatly indebted to the referee for writing a MAGMA program that helped us analyze the bitangents to the curves given by Kaplan.  We would also  like to thank Robert Lazarsfeld for introducing them at the University of Michigan and  Zachary Scherr for discussing this problem with us.

\section{Degree Two del Pezzo Surfaces}\label{knecht:seclabel2}
From now on we will work over a finite field of odd characteristic, $\fq$.  As stated above,  a split degree two del Pezzo surface $X_2$ can be thought of as the blow-up of the plane at seven points in general position or a degree four hypersurface in weighted projective space $\P(2,1,1,1)$.  The weighted projective space we are considering has variables $w,x,y,z$ where the $w$ is given weight two and the other variables are weight one.   This means for our surface is that we can write $X_2$ as the zero set of an equation of the form $$w^2= f_4(x,y,z)$$ where $f_4$ is a homogeneous degree four polynomial in three variables with coefficients in $\fq$.

This representation of $X_2$ yields a third way of thinking of degree two del Pezzo surfaces.  The polynomial $f_4(x,y,z)$ defines a quartic curve $Q$ in the plane $\P^2,$ and $X_2$ is a double cover of the plane branched over $Q$.  Thus, there is a surjective map $\varphi:X_2 \rightarrow \P^2$ such that the preimage of a point in the plane is two points in $X_2$ unless the point is contained in the quartic $Q$, in which case the preimage is one point.  The map  $\varphi$ is the rational map given by the anticanonical
linear system, $|-K_{X_d}|$, and happens to be a morphism in this case \cite[III.3.5]{MR1440180}.   If $L$ is a line in the plane, then by Bezout's Theorem $L$ intersects $Q$ at four points counting multiplicity and we write $L\cap Q= P_1+P_2+P_3+P_4$.  The line $L$ is \textit{tangent} to $Q$ at a point $P_1$ if we can say $L \cap Q = 2P_1 + P_2+ P_3$.  A \textit{bitangent} is a line $L$ such that $L \cap Q = 2P_1 + 2P_2$ or $L \cap Q = 4P_1 $.  When $L$ and $Q$  intersect only at one point $P_1$ we call that point a \textit{hyperflex}.  Over a field of characteristic greater than three, the number of hyperflexes is bounded by 12 \cite{MR812443}.  
There are only two quartic curves with exactly 12 hyperflexes, the Fermat quartic defined by  $x^4+y^4+z^4=0$ and the curve defined by $x^4+y^4+z^4 +3(x^2y^2 + x^2z^2 +y^2z^2)=0$  \cite{MR555703}.
Over $\F_9$, the Fermat quartic has 28 hyperflexes.

The fifty-six exceptional curves on $X_2$ come from the twenty-eight bitangents to the quartic curve $Q$.  
Each bitangent to $Q$ lifts to a pair of exceptional curves in $X_2$ that intersect transversely at two points of $Q$ or, in the case of a hyperflex,  are tangent to each other as pictured below.
\begin{center}
\includegraphics[scale=1]{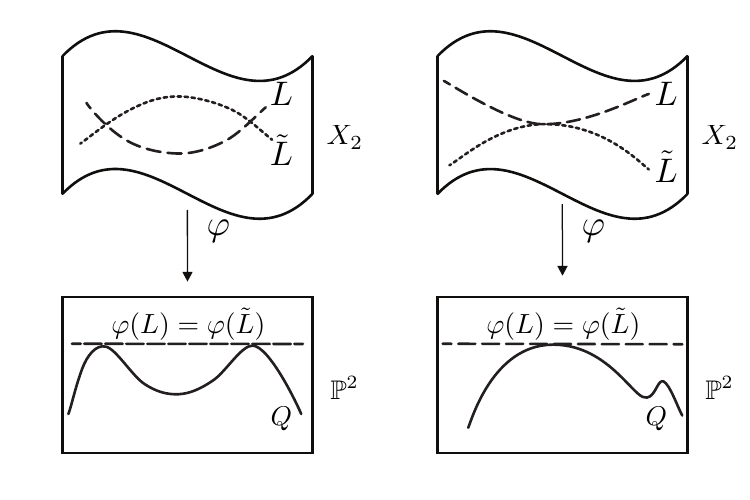}
\end{center}

Since the bitangents to $Q$ are lying in the plane $\P^2$, each bitangent is guaranteed to intersect each of the other 27 bitangents once.    Similarly to the case of cubic surfaces, now up to four bitangents can intersect at once \cite[Proof of Lemma 4.1]{MR2579393}.  This corresponds to four of the exceptional curves in $X_2$ intersecting at one point.  We call such a point in $X_2$ a \emph{generalized Eckardt point}.  As with cubic surfaces, a point on $X_2$ where three exceptional curves meet is called an Eckardt point.  In order to find a formula  for $L_{2,q}$ similar to Hirschfeld's formula for $L_{3,q}$ we need to establish a notation for the points on the bitangent lines to $Q$.

Let $L_i$ be a bitangent  to $Q$ for $1\leq i \leq 28$ and let:\\
$h_i$ = number of points on $L_i$ where exactly 4 bitangents meet,\\
$e_i$ = number of points on $L_i$ where exactly 3 bitangents meet,\\
$f_i$ = number of points on $L_i$ where  exactly 2 bitangents meet,\\
$g_i$ = number of points on $L_i$ that are not on any other bitangents,\\
$c_i$= number of times $L_i$ intersects $Q$ over $\F_q$.  Thus $c_i \in  \{0,1,2\}$.\\
Then,
\begin{equation}\label{line} |L_i(\fq)|= q+1= h_i+e_i+f_i+g_i+c_i \end{equation}  and 
\begin{equation}\label{int}27= 3h_i+2e_i+f_i.\end{equation}

Now let 
$h= \frac{1}{4} \sum h_i,\; e= \frac{1}{3} \sum e_i,\; f= \frac{1}{2} \sum f_i,\; g = \sum g_i,\; c= \sum c_i.$
Since each $L_i$ will lift to two exceptional curves in $X_2$ that intersect only over  $\fq$ points of $Q$,
 we get the formula
\begin{equation}\label{pts}L_{2,q}= 2h+2e+2f+2g+c.\end{equation}  Notice that we do not double count the $c$ points because those are places where the map $\varphi$ is one-to-one. Also notice that in order for $X_2$ to have a chance at fullness, we need $c= |Q(\F_q)|\leq 56$.
 
 We can use equations  (\ref{line})  and  (\ref{int})  to turn (\ref{pts}) into a formula which depends on only $q$, the number of points on the quartic $Q$,  and  the number of generalized and classical Eckardt points.  To do this we first manipulate  (\ref{int}) to get the equality $f_i=27-3h_i-2e_i$ which leads to \begin{equation}\label{f}2f=28(27)-12h-6e.  \end{equation}  Then combining $g_i=q+1-h_i-e_i-f_i-c_i$ from (\ref{line})  with  (\ref{f}) we get the following formula for $g$,
\begin{equation}\label{g} 2g= 28(2q-52)+16h+6e- 2|Q(\F_q)|.
\end{equation}
Using equations (\ref{f}) and (\ref{g}) to eliminate the $f$ and $g$ in (\ref{full}) yields the desired formula
 $$L_{2,q}= 6h+2e+28(2q-25)-|Q(\fq)|.$$  
 
 In order for $X_2$ to be full, the equation $$(q^2+8q+1)-(6h+2e+28(2q-25)-|Q(\fq)|)=0$$ must be satisfied.  Or we can solve 
 \begin{equation}\label{full}(q-24)^2+|Q(\fq)| =6h+2e-125.\end{equation}

\begin{lemma}  Full degree two del Pezzo surfaces can only exist over $\F_q$  when $9 \leq q \leq 37$.
\end{lemma}
\begin{proof}  In order to find a lower bound for the size of the field, we must find the smallest field in which a degree two del Pezzo can split.  We discover this bound by looking at how the bitangents of a quartic plane curve intersect.  Since each of the bitangents must intersect the 27 others and at most four can intersect simultaneously, each bitangent must contain at least 9 points.  The number of points on a line defined over $\fq$ is $q+1$, hence $q\geq 9.$

The upper bound is found by analyzing equation (\ref{full}).  There are at most 126 generalized Eckardt points on a degree two del Pezzo surface since each bitangent line contains at most $9$ of them.   Hence, $2h \leq \frac{28*9}{2}=126, h\leq 63.$  There are also at most 242 Eckardt points on a degree 2 del Pezzo surface since each bitangent can contain at most 13 Eckardt points and not all 28 of them can.  Even so, assuming all 28 lines contain 13 Eckardt points yields $e\leq 121$.
Thus, we can see that the maximum value for the right hand side of  (\ref{full}) is 253 when $h=63, e=0$.  In order to maximize the $q$ on the left hand side of (\ref{full}), we can assume that the quartic $Q$ does not have any $\fq$-rational points.  Thus we see that in order for a degree two del Pezzo surface defined over $\fq$ to have the possibility of being full, we need $(q-24)^2 \leq 253$ or $q\leq 37$.  When $q=37$, the Weil bounds tell us that $Q$ will contain at least 2 and at most 74 rational points, so the case $h=63, e=0, |Q(\fq)|=0$ does not apply.  In this case (\ref{full}) becomes $296\leq 6h+2e \leq 350$.   One solution to this equality is $h=21, e=85$, which we have not yet ruled out as a possibility.

\end{proof}

The degree two del Pezzo surfaces we experiment with all contain generalized Eckardt points.  Over small finite fields, split surfaces always have generalized Eckardt points.   But  Kaplan  \cite[Prop 76]{nathan} has found two surfaces over $\F_{23}$ that do not contain these points.

\begin{lemma}  Over $\F_9, \F_{11}$ and $\F_{13},$ split degree two del Pezzo surfaces contain generalized Eckardt points.
\end{lemma}
\begin{proof}
Over a field with $q$ elements, every bitangent to the quartic $Q$ contains exactly $q+1$ $\fq$-rational points.   
When $q=9$, $h_i=9$  for each $1\leq i\leq 28$, so $X_2$ contains $2\cdot (\frac{28\cdot 9}{4})=126$ generalized Eckardt points.
When $q=11$,  there are 12 points on each bitangent so $h_i\geq 1$ for each $1\leq i \leq 28$.  Thus $X_2$ contains at least  $2\cdot (\frac{28\cdot 1}{4})=14$ generalized Eckardt points.

Now suppose by way of contradiction that  the size of the ground field is 13 and the del Pezzo does not contain any generalized Eckardt points.  Then, every bitangent to $Q$ must intersect 2 lines simultaneously at 13 points, $e_i=13, f_i=1$ for all $1\leq i \leq 28$. This  means that the number of Eckardt points on the corresponding del Pezzo surface  is $2\cdot(\frac{28\cdot 13}{3})=\frac{364}{3} \notin \mathbb{Z}$.   Thus, the surface must contain a generalized Eckardt point.
\end{proof}

\section{Kuwata Curves}\label{knecht:seclabel3}
One reason we restrict ourselves to surfaces with generalized Eckardt points  is that over an algebraically closed field, a quartic curve $Q$ admitting  a generalized Eckardt point is equivalent to $Q$ admitting an involution  \cite[Theorem 5.1]{MR2153953}.  Thus there is nice symmetry between the lines.   Kuwata gives explicit equations for the quartic curves which admit a pair of commuting involutions, and these are the surfaces we study \cite{MR2153953}.

\begin{theorem}[Kuwata]  Let $Q$ be a quartic plane curve with a pair of commuting involutions.  By a change of coordinates, possibly defined over a field extension, $Q$ can be defined by an equation of the form
$$a_1x^4+a_2y^4+a_3z^4+a_4x^2y^2+a_5y^2z^2+a_6z^2x^2=0.$$
\end{theorem}
Kuwata uses the equation above to find a family of quartic curves whose  twenty-eight bitangent lines are all defined over the base field which does not need to be algebraically closed.  We call the curves $C_{\lambda\mu \nu  }$ defined below  \textit{Kuwata curves}  \cite{MR2153953}.
\begin{theorem}[Kuwata] 
Let $\lambda, \mu, \nu$ be three elements in $\F_q$ satisfying
$$(1-\lambda^2)(1-\mu^2)(1-\nu^2)(1-\mu^2\lambda^2)(1-\nu^2\lambda^2)(1-\mu^2\nu^2)(1-\lambda^2\mu^2\nu^2)\neq 0.$$
Let $C_{ \lambda\mu \nu}$ be the quartic plane curve given by
\begin{multline*}
((1- \mu^2\nu^2)  x^2+(1-\nu^2\lambda^2)y^2+    (1-\mu^2\lambda^2)z^2  )^2\\
  - 4(1-\lambda^2\mu^2\nu^2)\left( (1-\nu^2)x^2y^2 + (1-\lambda^2)y^2z^2+(1-\mu^2)z^2x^2  \right) =0.
\end{multline*}
The twenty-eight double tangent lines of $C_{\lambda\mu \nu } $ are given by
\begin{gather*}
x \pm \lambda y = 0,   \hskip 1cm  y \pm \mu x= 0,    \hskip 1cm z\pm \nu x = 0, \\
 x \pm \lambda z=0,   \hskip 1cm  y \pm \mu z =0 ,    \hskip 1cm  z \pm \nu y=0,\\
(1+ \mu \nu)x \pm (1+\nu \lambda)y \pm (1-\lambda \mu)z =0,\\
(1+ \mu \nu)x \pm (1-\nu \lambda)y \pm (1+\lambda \mu)z =0,\\
(1- \mu \nu)x \pm (1+\nu \lambda)y \pm (1+\lambda \mu)z =0,\\
(1- \mu \nu)x \pm (1-\nu \lambda)y \pm (1-\lambda \mu)z =0.
\end{gather*}

\end{theorem}

We use  the computer algebra system MAGMA to find all Kuwata curves over our desired fields.  Then we use the equations of the lines to determine the intersection form $(h, e, f, g, c)$ of each line.  With this information, we are able to find full degree two del Pezzo surfaces.   The following section lists the results we were able to find while experimenting with Kuwata curves in MAGMA.

Another recent result involving degree two del Pezzo surfaces with generalized Eckardt points is due  to Salgado, Testa,  and V{\'a}rilly-Alvarado \cite{MR3245139}.   Let $q_2, q_4$ be homogeneous polynomials of degrees two and
four respectively  such that the polynomial $q_2^2-4q_4$ has distinct roots.   When the characteristic of the ground field is not two,   the surface $X_2$ in the weighted projective space $\P(1, 1, 1, 2)$
with equation $w^2 = x^4 + q_2(y, z)x^2+ q_4(y, z)$
is a smooth del Pezzo surface of degree two with an involution given by $x\mapsto -x$ and contains  at least two generalized Eckardt points $[1,0,0, 1]$ and $[1,0,0,-1]$. 
In fact, every del Pezzo surface of degree two with a point
contained in four exceptional curves has an involution and is of the form described  above.

\section{Full Surfaces over Small Finite Fields}
In this section we present the results we found using MAGMA and Kuwata curves.  For each of the finite fields we list the full degree two del Pezzo surfaces and how each of the bitangents to the quartic $Q$ intersects the other bitangents in the form $(h,e,f,g,c) $ defined above.

Our first result on full degree 2 del Pezzo surfaces was first found by Hirschfeld  when he was studying full cubic surfaces\cite{MR0233272}.  He found that there is only one cubic surface over $\F_9$ with exactly one rational point not on the lines.  This point can be blown-up to create a full degree two del Pezzo. 

\begin{theorem}
Over the field $\F_9$ there exists a unique full degree two del Pezzo surface  up to isomorphism.   It is given by the equation $$w^2=x^4+y^4+z^4$$ and all of the lines on $X$  contain nine generalized Eckardt points. 
\end{theorem}
All bitangents  are all of type (9,0,0,0,1) and the quartic $Q$ has exactly 28 $\F_9$-rational points, all of which are hyperflexes.  
\begin{center}
\includegraphics[scale= .5]{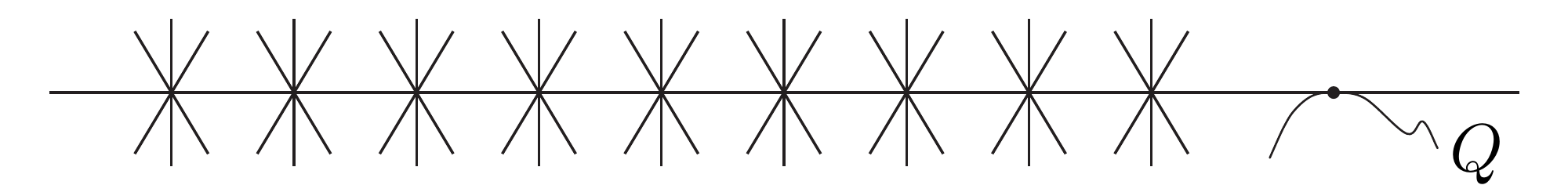}
\end{center}

\begin{theorem}
Over the field $\F_{11}$ there exists a unique full degree two del Pezzo surface up to isomorphism.   It is given by the equation $$w^2=x^4 + y^4 + z^4 + (x^2y^2 + x^2z^2 + y^2z^2).$$ \end{theorem}
 All of the bitangents to $Q$  are of type (3,9,0,0,0).  The quartic equation $x^4 + y^4 + z^4 + (x^2y^2 + x^2z^2 + y^2z^2)$ does not have any solutions over $\F_{11}$.
\begin{center}
\includegraphics[scale= .5]{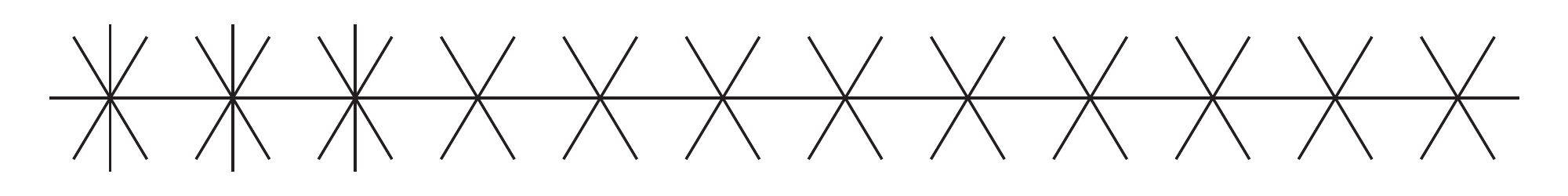}
\end{center}

\begin{theorem}
Over the field $\F_{13}$ there exist exactly  two full degree two del Pezzo surfaces up to isomorphism.  
\begin{equation}\label{f13_1}
w^2=x^4 + y^4 + z^4 + 8(x^2y^2 + x^2z^2 + y^2z^2)
\end{equation}
and
\begin{equation} \label{f13_2}
w^2=x^4 + y^4 + z^4 -x^2y^2.
\end{equation}
\end{theorem}
The bitangents corresponding to equation (\ref{f13_1}) come in two intersection types.  There are 24 bitangents of type (1,11,2,0,0) and 4 of type (3,9,0,0,2). The quartic curve $x^4 + y^4 + z^4 + 8(x^2y^2 + x^2z^2 + y^2z^2)
$ contains eight $\F_{13}$-rational points.
\begin{center}
\includegraphics[scale= .5]{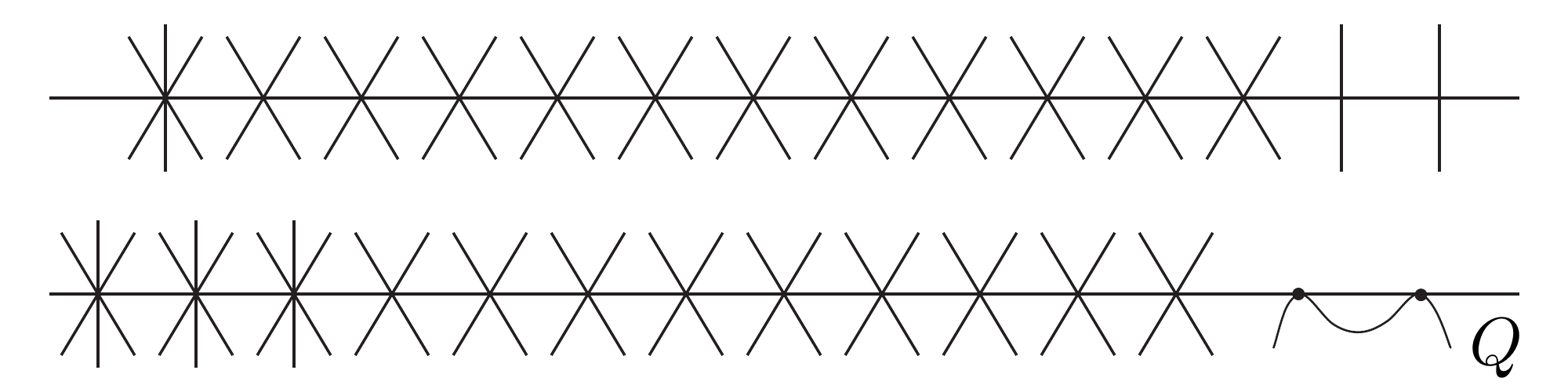}
\end{center}

The bitangents corresponding to equation (\ref{f13_2}) also come in two intersection types.  There are 24 bitangents of type (1,11,2,0,0) and 4 of type (1,12,0,0,1). The quartic curve $x^4 + y^4 + z^4 -x^2y^2$ contains four $\F_{13}$-rational points, all hyperflexes.

\begin{center}
\includegraphics[scale= .5]{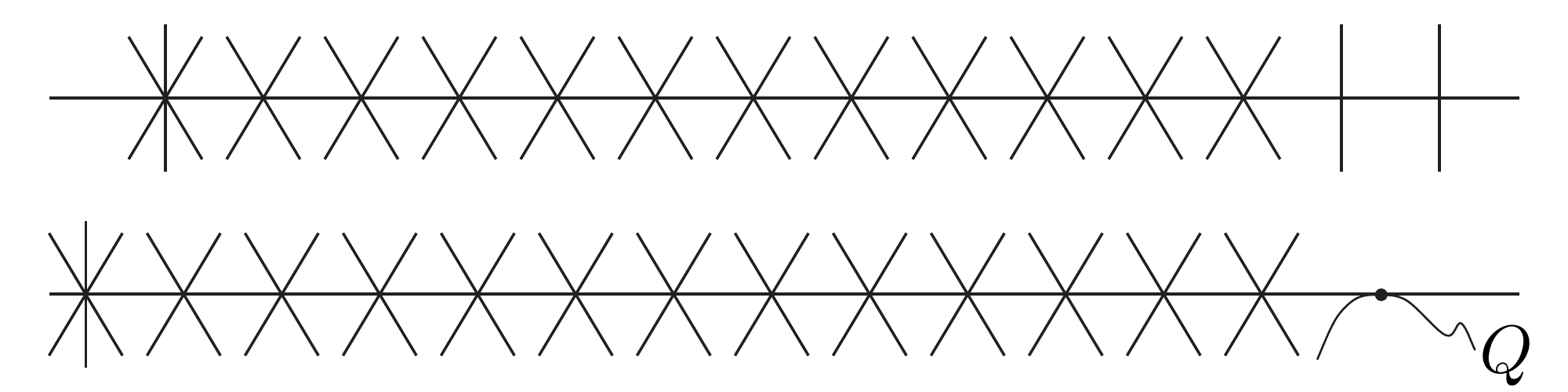}
\end{center}

We know that these are the only full degree two del Pezzos over the  fields $\F_9, \F_{11}, \F_{13}$ because Nathan Kaplan has found all the split degree two del Pezzo surfaces  over these fields \cite[Prop 70,71,72]{nathan}, and our results agree.  
\begin{corollary}  Over fields of size $9, 11,$ and $13$, a degree two del Pezzo surface is full if and only if it is split.
\end{corollary}

\begin{theorem}
Over the field $\F_{17}$ there are exactly six full degree two del Pezzo surfaces up to isomorphism.  

\begin{align} \label{f17_1}&w^2=x^4 + y^4 + z^4,\\
\label{f17_2}   &w^2=x^4 + y^4 + z^4 -x^2y^2 +7 x^2z^2 + 7y^2z^2,\\
 \label{f17_3}  &w^2=x^4 + y^4 + 2z^4 + 13x^2y^2 +6 x^2z^2 + 6y^2z^2,\\
\label{f17_4} &w^2=x^4 + y^4 + 2z^4 + 13x^2y^2 +x^2z^2 +y^2z^2, \\
\label{f17_5} &w^2=x^4+y^4+2z^4+5x^3y+15x^3z+13x^2y^2+5x^2yz+13x^2z^2+ \nonumber \\ 
& \hskip 1cm 13xy^3+ 15xy^2z+5xyz^2+5xz^3+2y^3z+13y^2z^2+12yz^3,\\
\label{f17_6} &w^2= x^4+2y^4+z^4 + 2x^3y+6x^2y^2+16x^2yz+ 11x^2z^2+5xy^3+\nonumber \\ 
& \hskip 1cm 7xy^2z+3xyz^2+6xz^3+6y^3z+10y^2z^2+8yz^3.
\end{align}

\end{theorem}

In  \cite[Prop 73]{nathan}, Kaplan gives seven isomorphism types for split degree two del Pezzo surfaces defined over $\F_{17}$.  His first four equations correspond to the first four full $X_2$'s we list above that come from Kuwata's equations.  His sixth equation is isomorphic to the Kuwata curve $C_{234}$ which is not full.  Thus, once the field has at least 17 elements, there is enough room for a degree 2 del Pezzo surface to be split but not full.  His fifth and seventh equations do not correspond to Kuwata curves, but are indeed full and are isomorphic to  the last two equations in the theorem.

The bitangents corresponding to equation (\ref{f17_1})  come in two intersection types.  
There are 12 bitangents of type (1,8,8,0,1), 16 of the type (3,3,12,0,0), and  $|Q(\F_{17})|=12$.
For equation (\ref{f17_2}) there are 12 bitangents of type (1,9,6,0,2), 12 of the type (1,7,10,0,0),  4 of type (3,6,6,3,0), and $|Q(\F_{17})|=24$.
For equation (\ref{f17_3}) there are 12 bitangents of type (1,9,6,0,2), 12 of the type (1,8,8,1,0),  4 of type (3,3,12,0,0), and $|Q(\F_{17})|=24$.
For equation (\ref{f17_4}) there are 8 bitangents of type (0,9,9,0,0), 8 of the type (1,8,8,1,0), 8 of type (1,7,10,0,0),  4 of type (1,9,6,0,2), and $|Q(\F_{17})|=8$.
 For equation (\ref{f17_5}) there are 4 bitangents of type (0,9,9,0,0), 12 of the type (0,10,7,1,0), 6 of type (1,7,10,0,0),  6 of type (1,9,6,0,2), and $|Q(\F_{17})|=12$.
 For equation (\ref{f17_6}) there is 1 bitangent of type (1,9,6,0,2), 2 of type (1,9,6,2,0), 8 of type (0,9,9,0,0), 10 of the type (0,10,7,1,0), 1 of type (1,7,10,0,0),  2 of type (0,10,7,0,1), 4 of type (0, 11, 5, 0, 2), and  $|Q(\F_{17})|=12$.

\begin{theorem}
Over the field $\F_{19}$ there exist  five full degree two del Pezzo surfaces up to isomorphism.  

\begin{align} \label{f19_1}& w^2=x^4 + y^4 + z^4+  4x^2y^2 +4x^2z^2+5 y^2z^2, \\
\label{f19_2} &w^2=x^4 + y^4 + z^4  -x^2y^2 +7 x^2z^2 + 7y^2z^2,\\
\label{f19_3} &w^2= x^4+y^4+z^4+12x^3y+2x^3z+15x^2y^2+14x^2yz+11x^2z^2+\nonumber\\
& \hskip .6cm 7xy^3+10xy^2z+5xyz^2+10xz^3+2y^3z+13y^2z^2+8yz^3,\\
\label{f19_4} &w^2= x^4+y^4+z^4+11x^3y+14x^3z+18x^2y^2+10x^2yz+9x^2z^2+\nonumber 
 \\ & \hskip .6cm 11xy^3 +16xy^2 z+18xyz^2+5xz^3+3y^3z+10y^2z^2+2yz^3, \\
\label{f19_5} &w^2=  x^4+y^4-z^4+x^3y+9x^3z+7x^2y^2+16x^2yz+x^2z^2+\nonumber  \\ 
& \hskip .6cm xy^3+7xy^2z+5xyz^2+6y^3z+11y^2z^2.
\end{align}
\end{theorem}

In  \cite[Prop 75]{nathan}, Kaplan gives fourteen isomorphism types for split degree two del Pezzo surfaces defined over $\F_{19}$.  His $8^{th}$ and $12^{th}$ equations correspond to the equations  (\ref{f19_1}) and  (\ref{f19_2})  in the list above.  His $5^{th}, 6^{th}, 10^{th}, 13^{th}$, and $14^{th}$ equations are isomorphic to the five other Kuwata curves and are not full.   Kaplan's $1^{st}$, $2^{nd}$, $9^{th}$ and $11^{th}$  equations are not Kuwata curves and are not full.  But his $3^{rd}$, $4^{th}$, and $7^{th}$ curves define full del Pezzos that are not Kuwata and are isomorphic to the last three equations in the theorem above.

The bitangents corresponding to the quartic in equation (\ref{f19_1}) come in five different types.  
There are 8 bitangents of type (1,7,10,2,0), 4 of type (1,7,10,0,2),   8 of the type (0,7,13,0,0), 4 of type (1,8,8,1,2), 4 of type (1,8,8,3,0), and $|Q(\F_{19})|=16$.
For equation (\ref{f19_2}) there are 12 bitangents of type (1,6,12,1,0),  12 of type (1,5,14,0,0),  4 of the type (3,6,6,3,2), and $|Q(\F_{19})|=8$.
For equation (\ref{f19_3}) there are 2 bitangents of type (0,10,7,3,0), 2 of type (0,10,7,1,2), 8 of type (0,9,9,2,0), 4 of type (0,9,9,0,2), 2 of type (0,9,9,1,1), 6 of type (0,8,11,1,0), 2 of type (1,6,12,1,0), 1 of type (1,9,6,2,2), 1 of type (1,5,14,0,0), and $|Q(\F_{19})|=16$.
For equation (\ref{f19_4}) there are 2 bitangents of type (0,10,7,2,1), 4 of type (0,10,7,1,2), 5 of type (0,7,13,0,0), 5 of type (0,9,9,2,0), 8 of type (0,8,11,1,0), 1 of type (1,7,10,0,2), 1 of type (1,9,6,4,0), 2 of type (1,7,10,2,0), and $|Q(\F_{19})|=12$.
For equation (\ref{f19_5}) there are 6 bitangents of type  (0,7,13,0,0), 6 of type (0,8,11,1,0), 6 of type (0,9,9,2,0),  6 of type (1,7,10,2,0), 3 of type (1,7,10,0,2), 1 of type (3,3,12,0,2), and $|Q(\F_{19})|=8$.

\begin{theorem}
Over the field $\F_{23}$ there exist at least two full degree two del Pezzo surfaces up to isomorphism.  

\begin{equation} \label{f23_1}w^2=x^4 + y^4 + z^4+  6(x^2y^2 +x^2z^2+ y^2z^2) \end{equation}
and
 \begin{equation} \label{f23_2}w^2=x^4 + y^4 + z^4 - (x^2y^2 + x^2z^2 + y^2z^2). \end{equation}
\end{theorem}
The bitangents corresponding to the quartic in equation (\ref{f23_1}) come in three different types.  
There are 12 bitangents of type (1,16,12,3,2), 12 of type (1,5,14,4,0), 4 of type (3,3,12,6,0), and $|Q(\F_{23})|=24$.
For equation (\ref{f23_2}) all 28 bitangents are of type (3,0,18,3,0)  and $Q(\F_{23})=\emptyset$.

There are twelve other isomorphism classes of Kuwata curves that correspond to split but not full degree two del Pezzo surfaces over $\F_{23}$.  In \cite[Prop 76]{nathan}, Kaplan says that there are at least 19 isomorphism classes of split degree two del Pezzo surfaces over $\F_{23}$.  He gives two classes that are not isomorphic to Kuwata curves because they have no non-trivial automorphisms.  Thus, Kaplan has found two split surfaces over $\F_{23}$ that contain no generalized Eckardt points.  His surfaces are not full.  Since Kaplan only gives the equations for two classes, and we know the equations for 14 Kuwata classes, there are still at least three isomorphism classes of split $X_2$'s to be found that could possibly be full.

Once the field size is larger than $23$, none of the del Pezzo surfaces corresponding to Kuwata curves are full.  Thus, a full $X_2$ defined over $\fq$ with $25\leq q \leq 37$ cannot have two commuting involutions in its automorphism group.  So we are left with the open question:
\begin{question}
Are there full degree two del Pezzo surfaces defined over $\F_{25}, \F_{27}, \F_{29}, \F_{31}$ and $\F_{37}$?
\end{question}
Another question we have yet to answer involves finding full degree two del Pezzos that do not come from Kuwata curves:
\begin{question}
Are there other full degree two del Pezzo surfaces defined over $\F_{23}$?
\end{question}
Once these questions are answered there remains the  harder open problem involving  smooth sextic hypersurfaces in the weighted projective space $\P(3,2,1,1)$:
\begin{question}
What are the full degree one del Pezzo surfaces?
\end{question}
This question is extremely computationally heavy because there are 240 exceptional curves to consider and $q^2+9q+1$ points on the surface.

\end{document}